\documentclass[a4paper]{amsart}

\usepackage{amssymb} % \nmid
\usepackage{amsmath}
\usepackage{amsthm}
\usepackage{dsfont, enumerate}
\usepackage{color}

\usepackage{verbatim}
\usepackage{url}
\newtheorem{thm}{Theorem}[section]
\newtheorem{lem}[thm]{Lemma}
\newtheorem{cor}[thm]{Corollary}

\theoremstyle{definition}

\usepackage{caption}
\usepackage{array} % add to the preamble to make the table
\newcolumntype{M}[1]{>{\centering\arraybackslash}m{#1}} % add to the preamble to make the table

\newenvironment{proofbold}[1]{\noindent\textbf{Proof of {#1}.}}{\hfill$\square$}

\newcommand{\ve}{\varepsilon}

\title{An $L^2$-bound for the Barban--Vehov weights%
}
\author{Olivier Ramar\'e}
\author{Sebastian Zuniga-Alterman}
\thanks{$^*$ The second author has been supported by the Finnish Centre of Excellence in Randomness
and Structures grant 346307 of the Academy of Finland,}
\newcommand{\Addresses}{{
  {\footnotesize
\ \\
  O.~RAMAR\'E - \textsc{CNRS / Institut de Math\'ematiques de Marseille, Aix Marseille Universit\'e,
U.M.R. 7373, Campus de Luminy, Case 907, 13288 MARSEILLE Cedex 9, France.}\\
  \texttt{olivier.ramare@univ-amu.fr}

\ \\
  S.~ZUNIGA ALTERMAN - \textsc{Department of Mathematics and Statistics, University of Turku, 20014 TURKU, Finland.}\\
  \texttt{szualt@utu.fi}

}}}

\begin{document}
% \address[O. Ramar\'e]{CNRS/ Institut de Math\'ematiques de Marseille, Aix 
% Marseille Universit\'e, U.M.R. 7373, Site Sud, Campus de Luminy, Case 907, 
% 13288 
% Marseille Cedex 9, France.}
% \email{olivier.ramare@univ-amu.fr}

% %\date{\sl January, the 7th of 2004}
% \subjclass[2010]{Primary: , Secondary: }

% \keywords{Class Field}

\maketitle

\begin{abstract}
  Let $\lambda$ the Barban--Vehov weights, defined in \eqref{delBVweights}. Let $X\ge z_1\ge100$ and  $z_2=z_1^\tau$ for some $\tau>1$. We prove that
    \begin{equation*}
    \sum_{n\le
      X}\frac{1}{n}\Bigl(\sum_{\substack{d|n}}\lambda_d\Bigr)^2
    \le
    f(\tau)\frac{\log X}{\log (z_2/z_1)},
  \end{equation*}
   for a completely determined function $f:(1,\infty)\to\mathbb{R}_{>0}$.
  In particular, we may take $f(2)=30$, saving more than a factor of $5$ on what was the best known result for $\tau=2$. Two related estimates are also provided for general $\tau>1$.
 
\end{abstract}

%%%%%%%%%%%%%%%%%%%%%%%%%
\section{Introduction and results}
%%%%%%%%%%%%%%%%%%%%%%%%%

In \cite{Barban-Vehov*68}, M.~Barban and P.~Vehov introduced the weights
\begin{equation}
  \label{delBVweights}
  \lambda_d = \lambda_d(z_2,z_1) =\mu(d)\frac{\log^+(z_2/d)-\log^+(z_1/d)}{\log(z_2/z_1)},
\end{equation}
defined for fixed $z_2>z_1>0$, where $\log^+=\max\{\log,0\}$. They discovered that the $L^2$-mean of $\sum_{d|n}\lambda_d$ is
remarkably small, even when $n$ is smaller than $z_2^2$. 

The Barban--Vehov weights were further studied for instance by
S.~Graham \cite{Graham*78}, Y.~Motohashi \cite[$\S$
1.3]{Motohashi*83} and M.~Haye Betah \cite{Betah*18}.

In particular, Y.~Motohashi found that, in order to use such weights
to obtain log-free density estimates, it is enough to upper bound
their harmonic $L^2$-average and that, in fact, it is sufficient to
bound
$\sum_{n\ge 1}\bigl(\sum_{\substack{d|n}}\lambda_d\bigr)^2/n^{1+\ve}$,
for any $\ve>0$. Here is our first result.
%%%%%%%%%%%%%
\begin{thm}
  \label{main1}
  Let $\ve>0$ and $z_1\ge 100$. Consider $z_2=z_1^\tau$ for some $\tau>1$. Then, 
  \begin{equation*}
    \sum_{n\geq 1}\frac{1}{n^{1+\ve}}\biggl(\sum_{\substack{d|n}}\lambda_d\biggr)^2    \le
    \frac{e^{\gamma\ve}}{\ve\log(z_2/z_1)}
    \frac{A(\tau+1)+[B-\tfrac{C }{z_1^{\ve}}+(B-\tfrac{C}{
      z_2^{\ve}}) \tau^2]\ve\log z_1}
    {\tau-1},
  \end{equation*}
  where $A=1.084$, $B=1.301$, $C=0.318$.
\end{thm}
%%%%%%%%%%%%%
The next is a result that enables one to compare Theorem \ref{main1} to previous work.
%%%%%%%%%%%%%
\begin{cor}
  \label{main2}
  When $X\ge z_1 \ge 100$ and $z_2=z_1^\tau$, for some $\tau>1$, we have
  \begin{equation*}
    \sum_{n\leq X}\frac{1}{n}\biggl(\sum_{\substack{d|n}}\lambda_d\biggr)^2
    \le
    3.09\frac{\log X}{\log(z_2/z_1)}\frac{1.084(\tau+1)+1.301(1+\tau^2)-0.116}
    {\tau-1}.
  \end{equation*}
  In particular, when $\tau=2$, the right hand side bound reads $30\log X/\log z_1$.
\end{cor}
%%%%%%%%%%%%%
The above estimate not only generalises the result of \cite{Betah*18},
but also saves more than a factor of $5$ on the special case considered
therein, namely when $\tau=2$. Note that a more precise explicit study
of these weights when $X\ge z_2^2$ has been the subject of
\cite{ZAlterman19-0} by the second author.

On the other hand, while studying zeros of the Riemann zeta function,
S. Graham also used the Barban--Vehov weights.  Here is a corollary
that may be compared to \cite[Lemma 9]{Graham*81}.
%%%%%%%%%%%%%
 \begin{cor}
  \label{main3}
  Let $X\ge z_1\ge 100$ and $z_2=z_1^\tau$, for some $\tau>1$. Then, for
  any $\alpha\in[1/2,1]$, we have
  \begin{equation*}
    \sum_{n\leq X}\frac{1}{n^{2\alpha-1}}\biggl(\sum_{\substack{d|n}}\lambda_d\biggr)^2
    \le
    3.09X^{2-2\alpha}\frac{\log X}{\log(z_2/z_1)}\frac{1.084(\tau+1)+1.301(1+\tau^2)-0.116}
    {\tau-1}.
  \end{equation*}
\end{cor}
%%%%%%%%%%%%%
As the reader will see, the proofs of the above results are remarkably simple, a fact that is
reflected thanks to the very moderate constants that appear throughout our bounds. It is noteworthy that, giving our historical remarks, it has taken quite a while to reach the reasonable bounds we present in Theorem \ref{main1}, Corollary \ref{main2} and Corollary \ref{main3}.

%The proof of this Theorem will be given at the end of chapter~\ref{Poids}.

%%%%%%%%%%%%%%%%%%%%%
\noindent\textbf{Notation.}
We shall use throughout the notation $\sigma=1+\ve$ --- we will use both notations interchangeably. Further, and in accordance with
earlier work on the subject, we use the family of functions
\begin{equation}
  \label{defcheckmqdeD}
  \check{m}_q(X,\sigma)=\sum_{\substack{n\le X\\
      (n,q)=1}}\frac{\mu(n)}{n^\sigma}
  \log\biggl
  (\frac{X}{n}\biggr).
\end{equation}
We shall also require the {\em Euler
  $\varphi_s$-function} defined, for any complex number $s$, by
  $\varphi_s(q)= q^s\prod_{p|q}\left(1-{1}/{p^s}\right)$. Finally, the variable $p$ always corresponds to a prime number.

%%%%%%%%%%%%%%%%%%%%%%%%% 
\section{Auxiliary results}
%%%%%%%%%%%%%%%%%%%%%%%%%

The following result may be found in \cite[Lemma 5.1]{Ramare*13d}.
%%%%%%%%%%
\begin{lem}
  \label{triv2} Let $\ve>0$. Then $\zeta(1+\ve)\le e^{\gamma \ve}/\ve$.
\end{lem}
%%%%%%%%%%
On the other hand, we have
%%%%%%%%%%%%%
\begin{lem}
  \label{mertensaux}
  Suppose that $Y\ge100$. Then
  \begin{equation*}
    \sum_{p\le
    Y}\frac{\log p}{p}
  \log\biggl(\frac{Y}{p}\biggr)\ge
  0.318 \log^2Y.
  \end{equation*}
  The above sum is in fact asymptotic to $\frac{1}{2}\log^2Y$, converging very slowly.
\end{lem}
%%%%%%%%%%%%%
%%%%%%%
\begin{proof}
We verify the claimed inequality for $Y\in[100,10^8]$ by using the
  Pari-GP script \texttt{GP/MertensAux.gp}.
  
  Suppose now that $Y\ge10^8$. By summation by parts,  $\sum_{p\le
    Y}\frac{\log p}{p}
  \log(\tfrac{Y}{p})$ equals
  \begin{align}\label{part1}
   &\log Y\sum_{p\leq 100}\frac{\log p}{p}-\sum_{p\leq 100}\frac{\log^2(p)}{p}+\int_{100}^Y(\vartheta(t)-\vartheta(100))\frac{1+\log(Y/t)}{t^2}dt\nonumber\\
  &\qquad\geq3.369\log Y-8.739+\int_{100}^Y(\vartheta(t)-\vartheta(100))\frac{1+\log(Y/t)}{t^2}dt,
  \end{align}
  where $\vartheta(t)=\sum_{p\leq t}\log p$ is the Chebyshev function.
  Now, by \cite[Thm. 6]{Schoenfeld*76}, we have that $
 \vartheta(t)> 0.998697\ t $ for any $t\geq 1155901$.
  On the other hand, one can verify that $\vartheta(t)> 0.835\ t$ provided that $100\leq t\leq 1155901$.
  Hence, 
   \begin{align}\label{part2}
    \int_{100}^Y(\vartheta(t)-\vartheta(100))\frac{1+\log(Y/t)}{t^2}dt&\geq\int_{100}^Y(0.835\ t-\vartheta(100))\frac{1+\log(Y/t)}{t^2}dt\nonumber\\
   &\geq \log^2(Y)\left[\frac{1}{2}+\frac{0.845}{\log Y}-\frac{5.058}{\log^2Y}\right].
  \end{align}
  By combining \eqref{part1} and \eqref{part2}, and then using that $Y\geq 10^8$, we arrive at
 \begin{equation*}
 \sum_{p\le
    Y}\frac{\log p}{p}
  \log\biggl(\frac{Y}{p}\biggr)\geq\log^2(Y)\left[\frac{1}{2}-\frac{13.8}{\log^2(10^8)}\right]\geq 0.376 \log^2Y.
 \end{equation*}
  As $0.376>0.318$, we derive the inequality of the statement.
  
Finally, thanks to the prime number theorem, $\vartheta(t)\sim t$, so we indeed have that $\sum_{p\le
    Y}\frac{\log p}{p}
  \log(\frac{Y}{p})\sim\frac{1}{2}\log^2Y$.
\end{proof}
%%%%%%%

%%%%%%%%%%%%%%%%%%%%%%%%% 
\section{Some lemmas involving $\check{m}_q(X,\sigma)$}
%%%%%%%%%%%%%%%%%%%%%%%%%
The following result is crucial for our improvement on \cite{Betah*18}. It can be found in \cite[Theorem 1.1]{Ramare-Zuniga*23-1}.
%%%%%%%%%%%%%%%%
\begin{lem}
  \label{checkmqdx}
  Let $X>0$. When $k\ge1$ and $\sigma\ge1$, we have
  \begin{equation*}
  0\le \sum_{\substack{n\le X\\ (n,q)=1}}\frac{\mu(n)}{n^{\sigma}}
  \log^k\biggl(\frac{X}{n}\biggr)
  \le
  1.00303\frac{q}{\varphi(q)}\bigl(
  k
  +(\sigma-1) \log X
  \bigr)\log^{k-1}X.
\end{equation*}
When $k\ge2$, we may replace $1.00303$ by 1.\\
\end{lem}

We also have the following result.

%%%%%%%%%%%%%%%%

%%%%%%%%%%%%%%%% 
\begin{lem}
  \label{BVRam}
  Let $Y\ge 100$. Then
  \begin{equation*}
    U(Y)=\sum_{\delta\le Y}
    \frac{\mu^2(\delta)  \varphi_{\sigma}(\delta)}{\delta^{2\sigma}}
    \check{m}_{\delta}\Bigl(\frac{Y}{\delta};\sigma\Bigr)^2
    \le
    \left(A+(\sigma-1)\left(B-\tfrac{C}{Y^\ve}\log Y\right)\right)\log Y.
  \end{equation*}
where $A=1.084$, $B=1.301$, $C=0.318$.
\end{lem}
%%%%%%%%%%%%%%%% 
  
%%%%%%%%%%%
\begin{proof}
  Each $\check{m}_\delta(y/\delta;\sigma)$ is non-negative, thanks to
  Lemma~\ref{checkmqdx} with $k=1$. Thus, we have the following estimation
  \begin{equation}\label{est}
    U(Y)
    \le
    1.00303\sum_{\delta\le
      Y}\frac{\mu^2(\delta)\varphi_{\sigma}(\delta)}{\delta^{2\sigma}}
      \frac{\delta}{\varphi(\delta)}
      \check{m}_{\delta}\Bigl(\frac{Y}{\delta};\sigma\Bigr)
    \biggl(1+\ve\log\bigg(\frac{Y}{\delta}\bigg)\biggr).
  \end{equation}
  Let us define the auxiliary non-negative multiplicative function $f$ on primes $p$ by
  $f(p)=\frac{1-p^{-\sigma}}{1-1/p}>0$ and on prime powers by
  $f(p^k)=0$ for $k\in\mathbb{Z}_{>1}$. We observe that  
    \begin{equation}\label{iden}
    \frac{\mu^2(\delta)\varphi_{\sigma}(\delta)}{\delta^{\sigma}}
    \frac{\delta}{\varphi(\delta)}
    =\mu^2(\delta)\sum_{a|\delta}f(a).
  \end{equation}
  On using identity \eqref{iden} on estimation \eqref{est}, and recalling definition \eqref{defcheckmqdeD}
, we derive
  \begin{align*}
    U(Y)
    &\le
    1.00303\sum_{a\le Y}\mu^2(a)\frac{f(a)}{a^{\sigma}}
      \sum_{\substack{b\le Y/a\\ (b,a)=1}}\frac{\mu^2(b)}{b^{\sigma}}
      \check{m}_{ab}\Bigl(\frac{Y}{ab};\sigma\Bigr)
    \biggl(1+\ve\log\bigg(\frac{Y}{ab}\bigg)\biggr)
    \\&\le
    1.00303
    \sum_{a\le Y}\mu^2(a)\frac{f(a)}{a^{\sigma}}
    \biggl(1+\ve\log\bigg(\frac{Y}{a}\bigg)\biggr)
    \sum_{\substack{d\le Y/a\\(d,a)=1}}
    \frac{\mu^2(d)}{d^{\sigma}}\log\biggl(\frac{Y/a}{d}\biggr)
    \sum_{bc=d}
    \mu(c)
    \\&\qquad
    -1.00303\ve
    \sum_{a\le Y}\mu^2(a)\frac{f(a)}{a^{\sigma}}
    \sum_{\substack{d\le Y/a\\ (d,a)=1}}
    \frac{\mu^2(d)}{d^{\sigma}}\log\biggl(\frac{Y/a}{d}\biggr)
    \sum_{bc=d}\mu(c)\log b.
  \end{align*}
  By M\"obius inversion, the innermost sum of the above first term reduces to $d=1$. Moreover, as $\sum_{bc=d}\mu(c)\log(b)=\Lambda(d)$, where $\Lambda$ is the von Mangoldt function, the innermost sum of the above second term is supported on prime numbers. Hence, 
  \begin{align*}
    \frac{U(Y)}{1.00303}
    \le&
      \sum_{a\le Y}\mu^2(a)\frac{f(a)}{a^{\sigma}}
    \biggl(1+\ve\log \bigg(\frac{Y}{a}\bigg)\biggr)\log\bigg(\frac{Y}{a}\bigg)
  \\
  &\qquad - \ve\,
    \sum_{a\le Y}\mu^2(a)\frac{f(a)}{a^{\sigma}}
      \sum_{\substack{p\le Y/a\\ (p,a)=1}} \frac{\log
    p}{p^{\sigma}}\log\bigg(\frac{Y/a}{p}\bigg).
  \end{align*}
  In the above inequality, we consider the two terms coming from $a=1$. Then, by using the inequality $p^{-\sigma}\geq y^{-\ve}p^{-1}$, valid whenever $p\leq y$, we derive
    \begin{align*}
    \frac{U(Y)}{1.00303}
    \le&
      (1+\ve\log Y)\log Y+\sum_{2\le a\le
    Y}\mu^2(a)\frac{f(a)}{a^{\sigma}}
    \biggl(1+\ve\log \bigg(\frac{Y}{a}\bigg)\biggr)\log\bigg(\frac{Y}{a}\bigg)
  \\
  &\qquad -
   \ve\,
          \sum_{\substack{p\le Y}} \frac{\log
    p}{p^{\sigma}}\log\bigg(\frac{Y}{p}\bigg)
    - \ve\,
    \sum_{2\le a\le
      Y}\mu^2(a)\frac{f(a)}{a^{\sigma}}
      \sum_{\substack{p\le Y/a\\ (p,a)=1}} \frac{\log
    p}{p^{\sigma}}\log\bigg(\frac{Y/a}{p}\bigg).
  \end{align*}
  Next, by trivially bounding the above last term and using Lemma~\ref{mertensaux}, we have
  \begin{align*}
    \frac{U(Y)}{1.00303}
    &\le
    (1+\ve\log Y)\log Y
    - \frac{0.318\ve}{Y^\ve}\log^2 Y
    \\&\quad+
    \biggl[-1+\prod_{p\ge2}\biggl(1+\frac{p^{-\ve}(1-p^{-\ve})}{p(p-1)}\biggr)\biggr]
    \biggl(1+\ve\log\left(\frac{Y}{2}\right)\biggr)\log\left(\frac{Y}{2}\right).
  \end{align*}
  Furthermore, by using that $t\in\mathbb{R}_{>0}\mapsto t^{-1}(1-t^{-1})$ has a maximum $1/4$, we deduce that
    \begin{align*}
    &-1+\prod_{p\ge2}\biggl(1+\frac{p^{-\ve}(1-p^{-\ve})}{p(p-1)}\biggr)
    \le -1+\biggl(1+\frac{2^{-\ve}(1-2^{-\ve})}{2}\biggr)\prod_{p\ge3}\biggl(1+\frac{1}{4p(p-1)}\biggr)\\
    &\qquad\le -1+\biggl(1+\frac{2^{-\ve}(1-2^{-\ve})}{2}\biggr)1.08
    \le 0.08
    +0.54\cdot 2^{-\ve}(1-2^{-\ve})
    ,
  \end{align*}
  so that
  \begin{align*}
    \bigl(0.08
    &+0.54\cdot 2^{-\ve}(1-2^{-\ve})\bigr)(1+\ve\log Y)
    \\&\le
    0.08 + \ve
    \biggl(0.08+0.54\cdot
    2^{-\ve}(1-2^{-\ve})+0.54\frac{2^{-\ve}(1-2^{-\ve})}{\ve\log
      Y}\biggr)\log Y\\
      %&\leq 0.08+\ve(0.08+\frac{0.54}{4}+\frac{0.54*\log 2}{\log 100})
    &\le 0.08+0.297\, \ve\log Y,
  \end{align*}
  where we have used that $t\in\mathbb{R}_{\geq 0}\mapsto (1-2^{-t})/t$ is decreasing and has as maximum $\log 2$ at $t=0$. 
  We conclude the result by observing that
  \begin{align*}
    \frac{U(Y)}{1.00303}
    &\le
    (1+\ve\log Y)\log Y
    - \frac{0.318\ve}{Y^\ve}\log^2Y+
    \biggl(0.08+0.297\ve\log Y\biggr)\log\left(\frac{Y}{2}\right)\\
    &\le 1.08\log Y+\ve\log^2Y\bigg(1.297-\tfrac{0.318}{Y^\ve}\bigg).
  \end{align*}
\end{proof}
%%%%%%%%%%%

%%%%%%%%%%%%%%%%%%%%%%%%%%%%%%%%%
\section{Main result}
%%%%%%%%%%%%%%%%%%%%%%%%%%%%%%%%%

%%%%%%%%%%
\begin{proofbold}{Theorem~\ref{main1}}
Recall definition \eqref{delBVweights}. Let us define
\begin{equation}
  \label{defSigma}
    \Sigma(\sigma)=\Sigma(\sigma,z_2,z_1)=\sum_{n}\frac{1}{n^\sigma}
    \Bigl(\sum_{\substack{d|n}}\lambda_d\Bigr)^2.
  \end{equation}
By expanding the square, we observe that
\begin{equation}\label{eq_0}
   \Sigma(\sigma)
  =
  \sum_{d,e}\lambda_{d}\lambda_{e}
  \sum_{[d,e]|n}\frac{1}{n^{\sigma}}
  =
  \zeta(\sigma)
  \sum_{d,e}
  \frac{\lambda_{d}\lambda_{e}}{[d,e]^{\sigma}},
\end{equation}
where the variables $d$ and $e$ may be restricted to $d,e\le z_2$.

Let us use now the Selberg diagonalisation process. For any pair of square-free numbers $d,e$, we have
\begin{equation}
  \label{eq:12}
  \frac{d^{\sigma}e^{\sigma}}{[d,e]^{\sigma}}
  =
  (d,e)^{\sigma}
  =\sum_{\delta|(d,e)}\varphi_{\sigma}(\delta).
\end{equation}
As the support of $\lambda$ is contained in the set of square-free
numbers, we may derive from \eqref{eq:12} the following expression 
\begin{equation}
  \label{eq:13}
  \sum_{d,e}
  \frac{\lambda_{d}\lambda_{e}}{[d,e]^{\sigma}}
  =
  \sum_{\delta}
  \frac{\varphi_{\sigma}(\delta)}{\delta^{2\sigma}}
  \biggl(
  \sum_{\substack{\ell\\(\ell,\delta)=1}}
  \frac{\lambda_{\delta\ell}}{\ell^{\sigma}}
  \biggr)^2.
\end{equation}
Thus, by definition \eqref{delBVweights} and equation \eqref{eq_0}, we
obtain
  \begin{equation}
    \label{goodstep}
   \Sigma(\sigma)
  \le
  \frac{\zeta(\sigma)}{\log^2(z_2/z_1)}
  \sum_{\delta\le z_2}\frac{\mu^2(\delta)
    \varphi_{\sigma}(\delta)}{\delta^{2\sigma}}
  \biggl[
  \check{m}_{\delta}\Bigl(\frac{z_2}{\delta};\sigma\Bigr)
  -\check{m}_{\delta}\Bigl(\frac{z_1}{\delta};\sigma\Bigr)\biggr]^2,
\end{equation}
where we have used the definition~\eqref{defcheckmqdeD}.
Furthermore, as $\check{m}_{q}(X;\sigma)\ge0$, thanks to Lemma \ref{checkmqdx}, we find that
\begin{equation*}
  \bigg[\check{m}_{\delta}\bigg(\frac{z_2}{\delta};\sigma\bigg)
  -\check{m}_{\delta}\bigg(\frac{z_1}{\delta};\sigma\bigg)\bigg]^2
  \le \check{m}^2_{\delta}\bigg(\frac{z_2}{\delta};\sigma\bigg)
  +\check{m}^2_{\delta}\bigg(\frac{z_1}{\delta};\sigma\bigg)
\end{equation*}
with equality if and only if
$\check{m}_{\delta}({z_2}/{\delta};\sigma)\check{m}_{\delta}({z_1}/{\delta};\sigma)=0$.
Therefore, we see from \eqref{goodstep} that
\begin{equation*}
  \frac{\log^2(z_2/z_1)}{\zeta(\sigma)}\Sigma(\sigma)
  \le
  \sum_{\delta\le z_2}
  \frac{\mu^2(\delta)  \varphi_{\sigma}(\delta)}{\delta^{2\sigma}}
  \check{m}^2_{\delta}\Bigl(\frac{z_2}{\delta};\sigma\Bigr)
  +
  \sum_{\delta\le z_1}
  \frac{\mu^2(\delta)  \varphi_{\sigma}(\delta)}{\delta^{2\sigma}}
  \check{m}^2_{\delta}\Bigl(\frac{z_1}{\delta};\sigma\Bigr),
\end{equation*}
so that
\begin{equation*}
\Sigma(\sigma)\leq\frac{\zeta(\sigma)}{\log^2(z_2/z_1)}(U(z_2)+U(z_1))=\frac{\zeta(\sigma)}{\log(z_2/z_1)}\frac{U(z_2)+U(z_1)}{(\tau-1)\log z_1}.
\end{equation*}
Theorem~\ref{main1} follows by applying Lemma~\ref{BVRam} twice, recalling that $z_2=z_1^\tau$ and using
then Lemma~\ref{triv2}.

\end{proofbold}
\newline

With the help of \eqref{eq:13}, it is worth noticing that, for any $\ve>0$ and any weight $\lambda$, the sum $\sum_{d,e}
  \frac{\lambda_{d}\lambda_{e}}{[d,e]^{\sigma}}$ is always non-negative.

%%%%%%%%%%%%%%%%%%%%%%%%% 
\section{Corollaries}
%%%%%%%%%%%%%
\begin{proofbold}{Corollary~\ref{main2}} 
  By using Rankin's trick, we have 
  \begin{equation*}
    \Sigma(0)=\sum_{n\leq X}\frac{1}{n}\biggl(\sum_{\substack{d|n}}\lambda_d\biggr)^2
\le
    \sum_{n}\frac{X^\ve}{n^{1+\ve}}
    \Bigl(\sum_{\substack{d|n}}\lambda_d\Bigr)^2
  \end{equation*}
  for any parameter $\ve$ that we shall soon choose.
  Now, by using Theorem~\ref{main1}, we obtain the bound
    \begin{equation}\label{zero}
    \Sigma(0)\le
    \frac{X^\ve e^{\gamma\ve}}{\ve\log(z_2/z_1)}
    \frac{A(\tau+1)+[B-\tfrac{C }{z_1^{\ve}}+(B-\tfrac{C}{
      z_2^{\ve}}) \tau^2]\ve\log z_1}
    {\tau-1},
  \end{equation}
    where $A=1.084$, $B=1.301$, $C=0.318$.
    
  Set $\ve=1/\log X$. Since there is no further assumptions on $z_2$, we use the bound $-C
      z_2^{-\ve}\le 0$. Thus, by writing $u=\ve\log z_1$, we obtain
      \begin{align*}
    \Sigma(0)&\le
    \frac{ e^{1+\gamma/\log X}\log X}{\log(z_2/z_1)}
    \frac{A(\tau+1)+[B(1+\tau^2)-\tfrac{C }{e^u}]u}
    {\tau-1}\\
    &\le
    \frac{ e^{1+\gamma/\log 100}\log X}{\log(z_2/z_1)}
    \frac{A(\tau+1)+B(1+\tau^2)-\tfrac{C }{e}}
    {\tau-1}
  \end{align*}
  where we have used that $X\geq 100$ and $0<u\leq 1$.
    
  In the specific and common situation
  when $\tau=2$, unlike the general case, we may take advantage of the factor $C
      z_2^{-\ve}=Ce^{-2u}$ appearing in \eqref{zero}. Hence, by using again that $X\geq 100$ and $0<u\leq 1$, we finally obtain
       \begin{equation*}
    \Sigma(0,z_1^2,z_1)\le
    e^{1+\gamma/\log 100}[A(\tau+1)+B(1+\tau^2)-\tfrac{C }{e}-\tfrac{C\tau^2}{e}]\frac{\log X}{\log(z_2/z_1)}\leq \frac{29.18\log X}{\log(z_2/z_1)}.
  \end{equation*}
         \end{proofbold}
%%%%%%%%%%%%%
\newline

\begin{proofbold}{Corollary~\ref{main3}}
  We simply write
  \begin{align*}
    \sum_{n\le X}\frac{1}{n^{2\alpha-1}}\biggl(\sum_{\substack{d|n}}\lambda_d\biggr)^2
    &=
      \sum_{n\le
      X}\frac{(\sum_{\substack{d|n}}\lambda_d)^2}{n}n^{2-2\alpha}\le
      X^{2-2\alpha}\sum_{n\le
      X}\frac{1}{n}\bigg(\sum_{\substack{d|n}}\lambda_d\bigg)^2,
  \end{align*}
  since $2-2\alpha\ge0$. Then the result follows by using Corollary~\ref{main2}.
\end{proofbold}
%%%%%%%%%%%%%

\bibliographystyle{plain}
\bibliography{Local}

%%%%%%%%%%%%%%%%%%%%%%%%%%%%%%%%%%%%
%%%%%%%%%%%%%%%%%%%%%%%%%%%%%%%%%%%%
%%%%%%%%%%%%%%%%%%%%%%%%%%%%%%%%%%%%

\Addresses

\end{document}